\documentclass[12pt]{amsart}
\usepackage{amsfonts,color,morefloats,pslatex}
\usepackage{amssymb,amsthm, amsmath,latexsym}

\newtheorem{theorem}{Theorem}
\newtheorem{lemma}[theorem]{Lemma}

\newtheorem{example}{Example}

\def\qi#1 {\fbox {\footnote {\ }}\ \footnotetext { From Qi: {\color{red}#1}}}

\begin{document}
\title[Infinite families of $2$-designs from linear codes]{Infinite families of $2$-designs from two classes of linear codes}

\author{Xiaoni Du}
\address{College of Mathematics and Statistics, Northwest Normal University, Lanzhou, Gansu 730070, China}
\curraddr{}
\email{ymldxn@126.com}
\author{Rong Wang}
\address{College of Mathematics and Statistics, Northwest Normal University, Lanzhou, Gansu 730070, China}
\curraddr{}
\email{rongw113@126.com}
\author{Chunming Tang}
\address{School of Mathematics and Information, China West Normal University, Nanchong, Sichuan 637002, China}
\curraddr{}
\email{tangchunmingmath@163.com}
\author{Qi Wang}
\address{Department of Computer Science and Engineering, Southern University of Science and Technology, Shenzhen, Guangdong 518055, China}
\curraddr{}
\email{wangqi@sustech.edu.cn}

\thanks{}
\date{\today}

\maketitle

\begin{abstract}
The interplay between coding theory and $t$-designs has attracted a lot of attention for both directions. It is well known that the supports of all codewords with a fixed weight in a code may hold a $t$-design. In this paper, by determining the weight distributions of two classes of linear codes, we derive infinite families of $2$-designs from the supports of codewords with a fixed weight in these codes, and explicitly obtain their parameters.

\end{abstract}

\section{Introduction}
Throughout this paper, let $p$ be an odd prime and $m$ be a positive integer. Let $\mathbb{F}_q$ denote the finite field with $q=p^m$ elements and $\mathbb{F}^*_q=\mathbb{F}_q\backslash\{0\}$. An $[n,k,d]$  {\em linear code} $\mathcal{C}$ over $\mathbb{F}_p$ is a $k$-dimensional subspace of $\mathbb{F}_p^n$ with minimum Hamming distance $d$, and is called a {\em cyclic code} if each codeword $(c_0,c_1,\ldots,c_{n-1})\in\mathcal{C}$ implies $(c_{n-1}, c_0, c_1, \ldots,
c_{n-2}) \in \mathcal{C}$. Any cyclic code $\mathcal{C}$ can be expressed as $\mathcal{C} = \langle g(x) \rangle$, where $g(x)$ is monic and has the least degree. The polynomial  $g(x)$ is called the  {\em generator polynomial} and $h(x)=(x^n-1)/g(x)$ is referred to as the {\em parity-check polynomial} of $\mathcal{C}$. The code, whose generator polynomial is $x^kh(x^{-1})/h(0)$, is called the {\em dual} of $\mathcal{C}$ and denoted by $\mathcal{C}^{\bot}$. Note that $\mathcal{C}^{\bot}$ is
an $[n,n-k]$ code. Furthermore, we define the  {\em extended code} $\overline{\mathcal{C}}$ of $\mathcal{C}$ to be the code
$$
\overline{\mathcal{C}}=\{(c_0, c_1, \ldots, c_n) \in \mathbb{F}_p^{n+1}: (c_0, c_1, \ldots, c_{n-1}) \in \mathcal{C} ~\mathrm{with} ~\sum^n_{i=0}c_i=0\}.
$$

Let $A_i$ be the number of codewords with Hamming weight $i$ in a code $\mathcal{C}$. The {\em weight enumerator} of $\mathcal{C}$ is defined by
$$1+A_1z+A_2z^2+\ldots+A_nz^n,$$
and the sequence $(1, A_1, \ldots, A_n)$ is called the {\em weight distribution} of the code $\mathcal{C}.$ If the number of nonzero $A_i$'s with $1 \leq i \leq n$ is exactly $w$, then we call $\mathcal{C}$ a {\em $w$-weight code}. Let $\mathbf{c}=(c_0, c_1, \ldots, c_{n-1})$ be a codeword in the code $\mathcal{C}.$ The {\em support} of $\mathbf{c}$ is defined by
$$Suppt(\mathbf{c})=\{0\leq i \leq n-1: c_i\neq 0\}\subseteq \{0, 1, \ldots, n-1\}.$$

Let $\mathcal{P}$ be a set of $v \geq1$ elements and $\mathcal{B}$ be a set of $k$-subsets of $\mathcal{P},$ where $k$ is a positive integer with $1\leq k \leq v$, and the size of $\mathcal{B}$ is denoted by $b$. Let $t$ be a positive integer with $t\leq k.$ If every $t$-subset of $\mathcal{P}$ is contained in exactly $\lambda$ elements of $\mathcal{B},$ then we call the pair $\mathbb{D}=(\mathcal{P}, \mathcal{B})$ a $t$-$(v,k,\lambda)$ design, or simply a {\em
$t$-design}. The elements of $\mathcal{P}$ are called {\em points}, and those of $\mathcal{B}$ are referred to as {\em blocks}. A $t$-design is {\em simple} when there is no repeated blocks in $\mathcal{B}$. A $t$-design is called {\em symmetric} if $v=b$ and {\em trivial} if $k=t$ or $k=v$. Hereafter, we restrict attention to simple $t$-designs with $t < k < v$. When $t\geq
2$ and $\lambda=1,$ we call the $t$-$(v, k, \lambda)$ design a  {\em Steiner system}. With simple counting argument, we have the following identity, which restricts the parameters of a $t$-$(v,k,\lambda)$ design.
\begin{equation}\label{condition}
  b {k \choose t} = \lambda {v \choose t}.
\end{equation}


Combinatorial $t$-designs have found important applications in coding theory, cryptography, communications and statistics.  The interplay between codes and $t$-designs is two-fold: On  one hand, a linear code over any {\em finite field} can be derived from the incidence matrix of a $t$-design and much progress has been made (see for example~\cite{AK92,DD15,Ton98,Ton07}). On the other hand, linear and nonlinear codes both may hold $t$-designs. For each $i$ with
$A_i\neq 0$, let $\mathcal{B}_i$ define the set of the supports of all codewords with weight $i$ in a code $\mathcal{C}$, where the coordinates of a codeword are indexed by $(0, 1, 2, \ldots, n-1)$. Let $\mathcal{P}=\{0, 1, \ldots, n-1\}$. Then the pair $(\mathcal{P},\mathcal{B}_i)$ might be a $t$-$(v,i,\lambda)$ design, where the parameter $\lambda$ can be accordingly determined using (\ref{condition}). In the literature some codes were used to
construct $2$-designs and $3$-designs~\cite{AK92,Ding15b,KP95,KP03,Ton98,Ton07}. Very recently, infinite families of $2$-designs and $3$-designs were obtained from several different classes of linear codes by Ding and Li~\cite{DL17}, and Ding~\cite{Ding182}. Some other constructions of $t$-designs can be found in~\cite{BJL99,CM06,MS771,RR10}.

Generally, if $\mathcal{C}$ is a cyclic code, the weight of each codeword can be expressed by certain exponential sums so that the weight distribution of $\mathcal{C}$ can be determined when the exponential sums could be computed explicitly (see~\cite{FY05,Vlugt95} and the references therein). Using this method, Feng and Luo~\cite{FL08} derived the weight distribution of cyclic code $\mathcal{C}_1$ with
length $n=q-1$ and parity-check polynomial $h_1(x)h_2(x)h_3(x),$ where $h_1(x),h_2(x)$ and $h_3(x)$ are the minimal polynomials of $\alpha,$ $\alpha^2$ and $\alpha^{p^l+1}$($l\geq1$ and gcd$(m, l)=1)$ over $\mathbb{F}_p,$ respectively,  and $\alpha$ is a primitive element of $\mathbb{F}_q$.

The main objective of this paper is to obtain $2$-designs from the following  two classes of linear codes ${\overline{{\mathcal{C}_1}^{\bot}}}^{\bot}$ and ${\overline{{\mathcal{C}_2}^{\bot}}}^{\bot}$.

\begin{eqnarray}\label{code-1}
{\overline{{\mathcal{C}_1}^{\bot}}}^{\bot}:=\{ (Tr (ax^{p^l+1}+bx^2+cx)_{x\in \mathbb{F}_q}+h ): a,b,c \in \mathbb{F}_q, h\in \mathbb{F}_p \}
\end{eqnarray}
and
\begin{eqnarray}\label{code-2}
{\overline{{\mathcal{C}_2}^{\bot}}}^{\bot}:=\{ (Tr (ax^{p^l+1}+bx)_{x\in \mathbb{F}_q}+h ): a,b \in \mathbb{F}_q, h\in \mathbb{F}_p \},
\end{eqnarray}
where  $Tr$ denotes the trace function from $\mathbb{F}_q$ onto $\mathbb{F}_p$, and $\mathcal{C}_2$ is the cyclic code with length $n$ and parity-check polynomial $h_1(x)h_3(x).$

The remainder of this paper is organized as follows. In Section \ref{section-2}, we introduce some notation and preliminary results on exponential sums, cyclotomic fields, affine invariant codes, which will be used in subsequent sections. In Section \ref{theorem}, we determine the weight distributions of two classes of linear codes by explicitly computing certain exponential sums. In Section \ref{section-4}, we then derive infinite families of $2$-designs and calculate their parameters from the two classes of codes in Section \ref{theorem}. Section \ref{section-5} concludes the paper.

\section{Preliminaries}\label{section-2}
In this section, we summarize some standard notation and basic facts on affine-invariant codes, exponential sums and cyclotomic fields.

\subsection{Some notation}
For convenience, we adopt the following notation unless otherwise stated in this paper.
\begin{itemize}
  \item  $p^*=(-1)^{\frac{p-1}{2}}p$.
\item  $m,l$ are  positive integers with $gcd(m, l)=1$, $q=p^m$.
\item $\mathcal{P}=\{0, 1, \ldots, n\}$ and   $n=p^m-1$.
\item $\zeta_p=e^{2\pi i/p}$ is a primitive $p$-th root of unity, where  $i=\sqrt{-1}$.
\item  $\eta$ and $\eta'$ are the quadratic characters of $\mathbb{F}_q^*$ and $\mathbb{F}_p^*,$ respectively. We extend these quadratic characters by setting $\eta(0)=0$ and $\eta'(0)=0$.
\item $Tr$ denotes the trace function from $\mathbb{F}_q$ onto $\mathbb{F}_p.$
\end{itemize}

\subsection{Affine-invariant codes and 2-designs}

We begin this subsection by the introduction of affine-invariant codes since the two classes of linear codes we investigate are both affine-invariant and will be proved to hold $2$-designs.

The set of coordinate permutations that map a code $\mathcal{C}$ to itself forms a group, which is referred to as the permutation automorphism group of $\mathcal{C}$ and denoted by $PAut(\mathcal{C})$. We define the affine group $GA_1(\mathbb{F}_q)$ by the set of all permutations
$$\sigma_{a,b}: x\mapsto  ax+b $$
of $\mathbb{F}_q$, where $a \in \mathbb{F}_q^*$ and  $b \in \mathbb{F}_q.$
An {\em affine-invariant} code is an extended cyclic code $\overline {\mathcal{C}}$ over $\mathbb{F}_p$ such that $GA_1(\mathbb{F}_q)\subseteq PAut(\overline{\mathcal{C}})$~\cite{HP03}.

The $p$-adic expansion of each $s\in\mathcal{P}$ is given by
$$
s=\sum^{m-1}_{i=0}s_ip^i,~ ~0\leq s_i\leq p-1 ,~0\leq i \leq m-1.
$$
For any $r=\sum^{m-1}_{i=0}r_ip^i \in\mathcal{P}$,  we say that $r\preceq s$ if $r_i \leq s_i$ for all $0\leq i\leq m-1$. Clearly, we have $r \le s $ if
$r\preceq s$.

For any integer $0\leq j< n$, the {\em $p$-cyclotomic coset} of $j$ modulo $n$ is defined by
$$C_j=\{jp^i \pmod n: 0 \le i \le \ell_j-1\}$$
where $\ell_j$ is the smallest positive integer such that $j\equiv jp^{\ell_j}\pmod n.$ The smallest integer in $C_j$ is called the {\em coset leader} of $C_j$. Let $g(x)=\prod_j\prod_{i\in C_j}(x-\alpha^i)$, where $j$ runs through some coset leaders of the $p$-cyclotomic cosets $C_j$ modulo $n.$  The set   $T=\bigcup_jC_j$ is referred to as the {\em defining set} of $\mathcal{C}$, which is the union of these $p$-cyclotomic cosets.

For certain applications, it is important to know whether a given extended primitive cyclic code $\overline{\mathcal{C}}$ is affine-invariant or not. The following lemma given by Kasami, Lin and Peterson~\cite{KLP67} answers the question by examining the defining set of the code.

\begin{lemma}~\cite{KLP67}\label{Kasami-Lin-Peterson}
  Let $\overline{\mathcal{C}}$ be an extended cyclic code of
  length $p^m$ over $\mathbb{F}_p$ with defining set $\overline{T}$. The code $\overline{\mathcal{C}}$ is affine-invariant if and
only if whenever $s \in   \overline{T}$  then $r \in \overline{T}$ for all $r \in \mathcal{P}$  with $r \preceq s$.
\end{lemma}

\begin{lemma}\label{The dual of an affine-invariant code}
  \cite{Ding18b} The dual of an affine-invariant code $\overline{\mathcal{C}}$ over $\mathbb{F}_p$ of length $n+1$
is also affine-invariant.
\end{lemma}

Affine-invariant codes are very attractive in the sense that they hold $2$-designs due to the following theorem.

\begin{theorem}\label{2-design}
  \cite{Ding18b} For each $i$ with $A_i \neq 0$ in an affine-invariant code $\overline{\mathcal{C}}$, the supports of the codewords of weight $i$ form a $2$-design.
\end{theorem}

Lemma \ref{The dual of an affine-invariant code} and Theorem \ref{2-design}  given above are very powerful tools in constructing $t$-designs from linear codes. We will employ them later in this paper. The following theorem given by Ding in~\cite{Ding18b} reveals the relationship of the codewords with the same support  in a linear code  $\mathcal{C}$, which will be used to calculate the parameters of  $2$-designs.

\begin{theorem}\label{design parameter}~\cite{Ding18b} Let $\mathcal{C}$ be a linear code over $\mathbb{F}_p$ with minimum weight $d$. Let $w$ be the largest integer with $w\leq n$ satisfying
 $${w-\lfloor\frac{w+p-2}{p-1}\rfloor}<d.$$
Let $\mathbf{c_1}$ and $\mathbf{c_2}$ be two codewords of weight $i$ with $d\leq i\leq w$ and $Suppt(\mathbf{c_1})=Suppt(\mathbf{c_2}).$ Then $\mathbf{c_1}=a\mathbf{c_2}$ for some $a\in \mathbb{F}_p^*$.
\end{theorem}

\subsection{Exponential sums}

An additive character of $\mathbb{F}_q$ is a nonzero function $\chi$ from $\mathbb{F}_q$ to the set of complex numbers of absolute value $1$ such that $\chi(x+y)=\chi(x)\chi(y)$ for any pair $(x,y) \in \mathbb{F}_q^2$. For each $u \in \mathbb{F}_q$, the function
$$\chi_u(v)=\zeta_p^{Tr(uv)},~v \in \mathbb{F}_q$$
denotes an additive character of $\mathbb{F}_q$. Since $\chi_0(v)=1$ for all $v \in \mathbb{F}_q$, we call $\chi_0$ the  {\em trivial} additive character of $\mathbb{F}_q$. We call $\chi_1$ the {\em canonical} additive character of $\mathbb{F}_q$ and we have  $\chi_u(x)=\chi_1(ux)$ for all $u\in\mathbb{F}_q$ ~\cite{LN97}.

To determine the parameters of codes ${\overline{{\mathcal{C}_1}^{\bot}}}^{\bot}$ and ${\overline{{\mathcal{C}_2}^{\bot}}}^{\bot}$ defined in Eqs.(\ref{code-1}) and (\ref{code-2}), we introduce the following functions
\begin{equation}\label{S(a.b.c)}
S(a,b,c)=\sum\limits_{x \in \mathbb{F}_q}\chi(ax^{p^l+1}+bx^2+cx),\quad a,b,c \in \mathbb{F}_q,
\end{equation}
and
\begin{equation}\label{S(a.b)}
S(a,b)=\sum\limits_{x \in \mathbb{F}_q}\chi(ax^{p^l+1}+bx), \quad a,b \in \mathbb{F}_q.
\end{equation}

The {\em Gauss sum} $G(\eta', \chi'_1)$ over $\mathbb{F}_p$ is defined by 
$$
G(\eta', \chi'_1)=\sum\limits_{v \in \mathbb{F}_p^*}\eta'(v)\chi'_1(v)=\sum\limits_{v \in \mathbb{F}_p}\eta'(v)\chi'_1(v),
$$
where $\chi'_1$ is the canonical additive characters of $\mathbb{F}_p$.  The following Lemmas \ref{Gauss}-\ref{value distribution} are essential to determine the values of Eqs.(\ref{S(a.b.c)}) and (\ref{S(a.b)}).

\begin{lemma}\label{Gauss}~\cite{LN97} With the notation above, we have
$$G(\eta', \chi'_1)={\sqrt{(-1)}^{(\frac{p-1}{2})^2}}\sqrt{p}=\sqrt{p^*}.$$
\end{lemma}

\begin{lemma}\label{Qu-Fea}\cite{DD15}
For each $y \in \mathbb{F}_p^*$,  $\eta(y)=1$ if $m\geq2$ is even, and $\eta(y)=\eta'(y)$ if  $m\geq1$ is odd.
\end{lemma}

\begin{lemma}\label{polynomial}~\cite{Coulter981} Let $f(x)=a^{p^l}x^{p^{2l}}+ax \in \mathbb{F}_q[x]$, gcd$(m, l)=1$ and $b \in \mathbb{F}_q.$ There are three cases.

$(1)$ If $m$ is odd, then $f(x)$ is a permutation polynomial over $\mathbb{F}_q$ and
$$S(a,b)=\sqrt{p^*}^m\eta(a)\chi_1(-ax_{a,b}^{p^l+1}).$$

$(2)$ If $m$ is even and $a^{\frac {q-1}{p+1}}\neq (-1)^{m/2},$ then $f(x)$ is a permutation polynomial over $\mathbb{F}_q$ and $$S(a,b)=(-1)^{m/2}p^{m/2}\chi_1(-ax^{p^l+1}_{a,b}).$$

$(3)$ If $m$ is even and $a^{\frac {q-1}{p+1}}=(-1)^{m/2}$, then $f(x)$ is not a permutation polynomial over $\mathbb{F}_q$. We have $S(a,b)=0$ when the equation $f(x)=-b^{p^l}$ is unsolvable, and
$$S(a,b)=-(-1)^{m/2}p^{m/2+1}\chi_1(-ax^{p^l+1}_{a,b})$$
otherwise.
In particular,
\begin{eqnarray*}
S(a,0)=\left\{
\begin{array}{ll}
\sqrt{p^*}^m\eta(a)   & \mathrm{if}\,\ m ~is ~odd,\\
(-1)^{\frac{m}{2}}p^{\frac{m}{2}}    & \mathrm{if}\,\ m ~is ~ even~and~ a^{\frac{q-1}{p+1}} \neq (-1)^{\frac{m}{2}}, \\
(-1)^{\frac{m}{2}+1}p^{\frac{m}{2}+1}    & \mathrm{if}\,\ m ~ is ~even~and~  a^{\frac{q-1}{p+1}} = (-1)^{\frac{m}{2}}.
\end{array}
\right.
\end{eqnarray*}
\end{lemma}
Notice that $x_{a,b}^{p^l+1}$ is a solution to the equation $f(x)=-b^{p^l}$. Moreover, $x_{a,b}^{p^l+1}$ is the unique solution when $f(x)$ is a permutation polynomial over $\mathbb{F}_q$.

\begin{lemma}\label{solution}~\cite{Coulter982} For $m$  even and gcd$(m, l)=1$, the equation $a^{p^l}x^{p^{2l}}+ax=0$ is solvable for $x \in \mathbb{F}_q^*$ if and only if $$a^{\frac{q-1}{p+1}}=(-1)^{\frac{m}{2}}.$$
In such cases there are $p^2-1$ nonzero solutions.
\end{lemma}

\begin{lemma}\label{value distribution}~\cite{FL08}
  For $m \geq 3$, gcd$(m, l)=1$, $\varepsilon=\pm1$, $0 \leq i \leq 2$ and $j\in \mathbb{F}_p^*,$  we define
\begin{eqnarray*}
n_{\varepsilon,i,j}=\left\{
\begin{array}{ll}
|\{(a,b,c)\in \mathbb{F}_q^3 : S(a,b,c)=\varepsilon\zeta^{j}_{p}p^{\frac{m+i}{2}} \}|   & \mathrm{if}\,\ m-i\,\mathrm{~is ~even}, \\
|\{(a,b,c)\in \mathbb{F}_q^3 : S(a,b,c)=\varepsilon\zeta^{j}_{p}\sqrt{p^*}p^{\frac{m+i-1}{2}} \}|     & \mathrm{if}\,\ m-i\,\mathrm{~is ~odd},
\end{array}
\right.
\end{eqnarray*}
and $ w=|\{(a,b,c)\in \mathbb{F}_q^3 : S(a,b,c)=0\}|$. Then the value distribution of the multiset $\{S(a,b,c) : a,b,c \in \mathbb{F}_q\}$ is given in Table \ref{1} when $m$ is odd and in Table \ref{2} when $m$ is even, respectively (see Tables~\ref{1} and~\ref{2} in Appendix I).
\end{lemma}

For clarity,  we denote the multiplicity of the lines $1-3$  in Table~\ref{1} by $n_{\pm1,0,0}, n_{1,0,1}, n_{-1,0,1},$ and  lines $1-4$ and $8-11$ in Table~\ref{2} by $n_{1,0,0}$, $n_{-1,0,0}$, $n_{1,0,1}$, $n_{-1,0,1}$ and $n_{1,2,0}$, $n_{-1,2,0}$, $n_{1,2,1}$, $n_{-1,2,1}$,  respectively.

\subsection{Cyclotomic fields}
We  state the following basic facts on Galois group of cyclotomic fields $\mathbb{Q}(\zeta_p)$ since  $S(a,b,c)$ and $S(a,b)$ are elements in $\mathbb{Q}(\zeta_p)$.

\begin{lemma}\label{Cyclotomic fields}~\cite{IR90}
  Let $\mathbb{Z}$ be the rational integer ring and $\mathbb{Q}$ be the rational field.

$(1)$ The ring of integers in $K=\mathbb{Q}(\zeta_p)$ is $\mathcal{O}_k=\mathbb{Z}[\zeta_p]$ and $\{\zeta^i_{p} : 1 \leq i \leq p-1\}$ is an integral basis of $\mathcal{O}_k.$

$(2)$ The filed extension $K/\mathbb{Q}$ is Galois of degree $p-1$ and the Galois group $Gal(K/\mathbb{Q})=\{\sigma_y : y \in (\mathbb{Z}/p\mathbb{Z})^*\},$ where the automorphism $\sigma_y$ of $K$ is defined by $\sigma_y(\zeta_p)=\zeta^y_{p}.$

$(3)$ $K$ has a unique quadratic subfield $L=\mathbb{Q}(\sqrt{p^*}).$ For $1 \leq y \leq p-1, \sigma_y(\sqrt{p^*})
=\eta'(y)\sqrt{p^*}.$ Therefore, the Galois group $Gal(L/\mathbb{Q})$ is $\{1,\sigma_\gamma\},$ where $\gamma$ is any quadratic nonresidue in $\mathbb{F}_p.$
\end{lemma}

\section{Weight distributions of two classes of linear codes}\label{theorem}
Now, we are ready to present the weight distributions of the codes ${\overline{{\mathcal{C}_1}^{\bot}}}^{\bot}$ and ${\overline{{\mathcal{C}_2}^{\bot}}}^{\bot}$ in Theorems \ref{weight1} and \ref{weight2}, respectively. We also use Magma programs to give some examples.

\begin{theorem}\label{weight1}
Let $m\geq 3$. The weight distribution of the cod ${\overline{{\mathcal{C}_1}^{\bot}}}^{\bot}$ over $\mathbb{F}_p$ with length $n+1$ and $dim({\overline{{\mathcal{C}_1}^{\bot}}}^{\bot})=3m+1$ is given in Table \ref{3} when $m$ is odd and in Table \ref{4} when $m$ is even, respectively.
\begin{table}
\begin{center}
\caption{The weight distribution of ${\overline{{\mathcal{C}_1}^{\bot}}}^{\bot}$ when $m$ is odd}\label{3}
\begin{tabular}{ll}
\hline\noalign{\smallskip}
Weight  &  Multiplicity   \\
\noalign{\smallskip}
\hline\noalign{\smallskip}
$0$  &  1 \\
$p^{m-1}(p-1)$  &  $ p(2p^{2m-1}-2p^{2m-2}+p^{2m-3}-p^{2m-4}+p^{m-1}$\\ &$-p^{m-2}+2)(p^m-1) $    \\
$p^{\frac{m-1}{2}}(p^{\frac{m+1}{2}}-p^{\frac{m-1}{2}}-p+1) $ & $ \frac{1}{2}p^{\frac{3m-3}{2}}(p^m-1)(p^{\frac{m-1}{2}}+1)  $     \\
$ p^{\frac{m-1}{2}}(p^{\frac{m+1}{2}}-p^{\frac{m-1}{2}}+p-1)$  &  $ \frac{1}{2}p^{\frac{3m-3}{2}}(p^m-1)(p^{\frac{m-1}{2}}-1)$     \\
$ p^{\frac{m-1}{2}}(p^{\frac{m+1}{2}}-p^{\frac{m-1}{2}}+1)$  &  $ \frac{p^m(p^m-1)(p^{m+2}-p^{m+1}-p^{m-2}+p^{\frac{m+1}{2}}-p^{\frac{m-3}{2}}+p^2)}{2(p+1)}$     \\
$ p^{\frac{m-1}{2}}(p^{\frac{m+1}{2}}-p^{\frac{m-1}{2}}-1)$  &  $ \frac{p^m(p^m-1)(p^{m+2}-p^{m+1}-p^{m-2}-p^{\frac{m+1}{2}}+p^{\frac{m-3}{2}}+p^2)}{2(p+1)}$     \\
$ p^{\frac{m+1}{2}}(p^{\frac{m-1}{2}}-p^{\frac{m-3}{2}}+1)$  &  $
\frac{1}{2}p^{m-2}(p^m-1)(p^{m-1}-1)/(p+1)$     \\
$ p^{\frac{m+1}{2}}(p^{\frac{m-1}{2}}-p^{\frac{m-3}{2}}-1)$  &  $
\frac{1}{2}p^{m-2}(p^m-1)(p^{m-1}-1)/(p+1)$     \\
$ p^m$  &  $ p-1$     \\
\noalign{\smallskip}
\hline
\end{tabular}
\end{center}
\end{table}

\begin{table}
\begin{center}
\caption{The weight distribution of ${\overline{{\mathcal{C}_1}^{\bot}}}^{\bot}$ when $m$ is even}\label{4}
\begin{tabular}{ll}
\hline\noalign{\smallskip}
Weight  &  Multiplicity   \\
\noalign{\smallskip}
\hline\noalign{\smallskip}
$0$  &  1 \\
$p^{m-1}(p-1)$  &  $ p(p^{2m-1}-p^{2m-2}+2p^{2m-3}-p^{m-2}+1)(p^m-1)$    \\
$p^{\frac{m-2}{2}}(p^{\frac{m+2}{2}}-p^{\frac{m}{2}}-p+1) $ & $ \frac{p^{m+2}(p^m-1)(p^m-p^{m-1}-p^{m-2}+p^{\frac{m}{2}}-p^{\frac{m-2}{2}}+1)}{2(p^2-1)}  $     \\
$ p^{\frac{m}{2}}(p^{\frac{m}{2}}-p^{\frac{m-2}{2}}-p+1)$  &  $ \frac{1}{2}p^{m-2}(p^m-1)(p^{\frac{m}{2}}-1)(p^{\frac{m-2}{2}}+1)/(p^2-1)$     \\
$ p^{\frac{m-2}{2}}(p^{\frac{m+2}{2}}-p^{\frac{m}{2}}+p-1)$  &  $ \frac{p^{m+2}(p^m-1)(p^m-p^{m-1}-p^{m-2}-p^{\frac{m}{2}}+p^{\frac{m-2}{2}}+1)}{2(p^2-1)} $     \\
$ p^{\frac{m}{2}}(p^{\frac{m}{2}}-p^{\frac{m-2}{2}}+p-1)$  &  $ \frac{1}{2}p^{m-2}(p^m-1)(p^{\frac{m}{2}}+1)(p^{\frac{m-2}{2}}-1)/(p^2-1)$     \\
$ p^{\frac{m-2}{2}}(p^{\frac{m+2}{2}}-p^{\frac{m}{2}}+1)$  & $ \frac{p^{m+2}(p^m-1)(p^m-p^{m-1}-p^{m-2}+p^{\frac{m}{2}}-p^{\frac{m-2}{2}}+1}{2(p+1)}$     \\
$ p^{\frac{m}{2}}(p^{\frac{m}{2}}-p^{\frac{m-2}{2}}+1)$ & $
\frac{p^{m-2}(p^m-1)(p^{m+2}-p^m+p^{m-1}+p^{\frac{m}{2}}-p^{\frac{m-2}{2}}-1)}{2(p+1)} $    \\
$ p^{\frac{m-2}{2}}(p^{\frac{m+2}{2}}-p^{\frac{m}{2}}-1)$  & $ \frac{p^{m+2}(p^m-1)(p^m-p^{m-1}-p^{m-2}-p^{\frac{m}{2}}+p^{\frac{m-2}{2}}+1)}{2(p+1)}$    \\
$ p^{\frac{m}{2}}(p^{\frac{m}{2}}-p^{\frac{m-2}{2}}-1)$ & $
\frac{p^{m-2}(p^m-1)(p^{m+2}-p^m+p^{m-1}-p^{\frac{m}{2}}+p^{\frac{m-2}{2}}-1)}{2(p+1)} $    \\
$ p^m$  &  $ p-1$     \\
\noalign{\smallskip}
\hline
\end{tabular}
\end{center}
\end{table}
\end{theorem}

One can see that the code is eight-weight if $m$ is odd and ten-weight if $m$ is even. The proof of Theorem~\ref{weight1} is put in Appendix II.

\begin{example}\label{example1}
If $(p,m)=(3,3)$, then the code ${\overline{{\mathcal{C}_1}^{\bot}}}^{\bot}$ has parameters $[27,10,9]$ and weight enumerator $1+78z^9+1404z^{12}+14040z^{15}+27300z^{18}+15444z^{21}+702z^{24}+80z^{27},$ which confirms  the results in Theorem \ref{weight1}.
\end{example}


\begin{example}\label{example3}
If $(p,m)=(3,4)$, then the code ${\overline{{\mathcal{C}_1}^{\bot}}}^{\bot}$ has parameters $[81,13,36]$ and weight enumerator $1+1440z^{36}+60120z^{45}+189540z^{48}+291600z^{51}+464640z^{54}+379080z^{57}+145800z^{60}+61200z^{63}+900z^{72}+2z^{81},$ which confirms the results in Theorem \ref{weight1}.
\end{example}


\begin{theorem}\label{weight2}
Let $m\geq 3$. The weight distribution of the code ${\overline{{\mathcal{C}_2}^{\bot}}}^{\bot}$ over $\mathbb{F}_p$ with length $n+1$ and $dim({\overline{{\mathcal{C}_2}^{\bot}}}^{\bot})=2m+1$ is given in Table \ref{5} when $m$ is odd and in Table \ref{6} when $m$ is even, respectively.
\begin{table}
\begin{center}
\caption{The weight distribution of ${\overline{{\mathcal{C}_2}^{\bot}}}^{\bot}$ when $m$ is odd}\label{5}
\begin{tabular}{ll}
\hline\noalign{\smallskip}
Weight  &  Multiplicity   \\
\noalign{\smallskip}
\hline\noalign{\smallskip}
$0$  &  1 \\
$p^{m-1}(p-1)$  &  $ p(p^{m-1}+1)(p^m-1) $    \\
$ p^{m-1}(p-1)+p^{\frac{m-1}{2}}$  &  $ \frac{1}{2}p^m(p-1)(p^m-1)$     \\
$ p^{m-1}(p-1)-p^{\frac{m-1}{2}}$  &  $ \frac{1}{2}p^m(p-1)(p^m-1)$     \\
$p^m$ & $ p-1 $     \\
\noalign{\smallskip}
\hline
\end{tabular}
\end{center}
\end{table}

\begin{table}
\begin{center}
\caption{The weight distribution of ${\overline{{\mathcal{C}_2}^{\bot}}}^{\bot}$ when $m$ is even}\label{6}
\begin{tabular}{ll}
\hline\noalign{\smallskip}
Weight  &  Multiplicity   \\
\noalign{\smallskip}
\hline\noalign{\smallskip}
$0$  &  1 \\
$p^{m-1}(p-1)$  &  $ p(p^{m-1}-p^{m-2}+1)(p^m-1) $    \\
$ p^{m-1}(p-1)-(-1)^{\frac{m}{2}}p^{\frac{m}{2}-1}(p-1)$  &  $ p^{m+1}(p^m-1)/(p+1)$     \\
$ p^{m-1}(p-1)+(-1)^{\frac{m}{2}}p^{\frac{m}{2}-1}$  &  $ p^{m+1}(p-1)(p^m-1)/(p+1)$     \\
$ p^{m-1}(p-1)+(-1)^{\frac{m}{2}}p^{\frac{m}{2}}(p-1)$  &  $ p^{m-2}(p^m-1)/(p+1)$     \\
$ p^{m-1}(p-1)-(-1)^{\frac{m}{2}}p^{\frac{m}{2}}$  &  $ p^{m-2}(p-1)(p^m-1)/(p+1)$     \\
$p^m$ & $ p-1 $     \\
\noalign{\smallskip}
\hline
\end{tabular}
\end{center}
\end{table}
\end{theorem}
One can see that the code ${\overline{{\mathcal{C}_2}^{\bot}}}^{\bot}$ is four-weight if $m$ is odd and six-weight if $m$ is even. Similarly, the proof of Theorem~\ref{weight2} is also put in Appendix II.


\begin{example}\label{example6}
If $(p,m)=(5,3)$, then the code ${\overline{{\mathcal{C}_2}^{\bot}}}^{\bot}$ has parameters $[125,7,95]$ and weight enumerator $1+31000z^{95}+16120z^{100}+31000z^{105}+4z^{125},$ which confirms the results in Theorem \ref{weight2}.
\end{example}

\begin{example}\label{example7}
If $(p,m)=(3,4)$, then the code ${\overline{{\mathcal{C}_2}^{\bot}}}^{\bot}$ has parameters $[81,9,45]$ and weight enumerator $1+360z^{45}+4860z^{48}+4560z^{54}+9720z^{57}+180z^{72}+2z^{81},$ which confirms the results in Theorem \ref{weight2}.
\end{example}


\section{Infinite families of $2$-Designs}\label{section-4}

In the following, we derive $2$-designs from  the codes presented in Section \ref{theorem}. To this end, we first prove that these codes are both affine-invariant.

\begin{lemma}\label{affine invariant}
The extended codes  $\overline{{\mathcal{C}_1}^{\bot}}$ and $\overline{{\mathcal{C}_2}^{\bot}}$ are affine-invariant.
\end{lemma}

\begin{proof}
We will prove the conclusion by Lemma \ref{Kasami-Lin-Peterson}.
The defining set $T$ of the cyclic code ${\mathcal{C}_1}^{\bot}$ is $T =C_1 \cup C_2 \cup C_{p^l+1}$. Since $0 \not \in T$, the defining set $\overline{T}$ of $\overline{{\mathcal{C}_1}^{\bot}}$ is given by $\overline{T} = C_1 \cup C_2 \cup C_{p^l+1} \cup \{0\}$.
Let $s \in \overline{T} $ and $r \in \mathcal{P}$. Assume that $r \preceq s$. We need to prove that $r \in \overline{T}$ by Lemma \ref{Kasami-Lin-Peterson}.

If $r=0,$ then obviously $r\in \overline{T}$. Consider now the case $r>0$. If $s \in  C_1 \cup C_2$, then the Hamming weight $wt(s) = 1.$ Since  $r \preceq s$, $wt(r) = 1.$ Consequently, $r \in C_1 \cup C_2 \subset  \overline{T}.$  If $s  \in C_{p^l+1}$, then the Hamming weight $wt(s) = 2.$ Since $r \preceq s$, either $wt(r) = 1$ or $r = s.$ In both cases, $r \in  \overline{T}.$ The desired conclusion then follows
from Lemma \ref{Kasami-Lin-Peterson}.

Similarly, we can prove that $\overline{{\mathcal{C}_2}^{\bot}}$ is affine-invariant.

Thus the proof is completed.
\end{proof}

By Lemmas \ref{The dual of an affine-invariant code} and \ref{affine invariant}, we know both ${\overline{{\mathcal{C}_1}^{\bot}}}^{\bot}$ and ${\overline{{\mathcal{C}_2}^{\bot}}}^{\bot}$ are affine-invariant. Thus we have the following result by  Theorems \ref{2-design}.

\begin{theorem}\label{$2-$design-1}
Let $m\geq 3$ be a positive integer. Then the supports of the codewords of weight $i>0$ in ${\overline{{\mathcal{C}_1}^{\bot}}}^{\bot}$ or  ${\overline{{\mathcal{C}_2}^{\bot}}}^{\bot}$ form a $2$-design, provided that $A_i\neq0.$
\end{theorem}

The parameters of the 2-designs derived from ${\overline{{\mathcal{C}_1}^{\bot}}}^{\bot}$ and ${\overline{{\mathcal{C}_2}^{\bot}}}^{\bot}$ are given in  Theorems \ref{parameter-1},~\ref{parameter-2} and \ref{parameter-3}, respectively. We only give the proof of Theorem \ref{parameter-1}, since Theorems \ref{parameter-2} and \ref{parameter-3} can be proved with similar arguments.


\begin{theorem}\label{parameter-1}
Let $m$ be an odd integer and $\mathcal{B}$ be the set of the supports of the codewords of ${\overline{{\mathcal{C}_1}^{\bot}}}^{\bot}$ with weight $i,$ where $A_i\neq 0.$ Then for $m\geq 5$, ${\overline{{\mathcal{C}_1}^{\bot}}}^{\bot}$ holds $2$-$(p^m, i, \lambda)$ designs for the following pairs:
\begin{itemize}
\item $(i, \lambda)=(p^m-p^{m-1}-p^{\frac{m+1}{2}}, p^{\frac{m-3}{2}}(p^{\frac{m-1}{2}}-p^{\frac{m-3}{2}}-1)(p^{m-1}-1)(p^m-p^{m-1}-p^{\frac{m+1}{2}}-1)/2(p^2-1)).$
\item $(i, \lambda)=(p^m-p^{m-1}-p^{\frac{m-1}{2}}(p-1), p^{m-2}(p^{m-1}-1)(p^m-p^{m-1}-p^{\frac{m+1}{2}}+p^{\frac{m-1}{2}}-1)/2).$
\item $(i, \lambda)=(p^m-p^{m-1}-p^{\frac{m-1}{2}}, p^{\frac{m-1}{2}}(p^{\frac{m+1}{2}}-p^{\frac{m-1}{2}}-1)(p^m-p^{m-1}-p^{\frac{m-1}{2}}-1)(p^{m+2}-p^{m+1}-p^{m-2}-p^{\frac{m+1}{2}}
    +p^{\frac{m-3}{2}}+p^2)/2(p^2-1)).$
\item $(i, \lambda)=(p^m-p^{m-1}, (p^m-p^{m-1}-1)(2p^{2m-1}-2p^{2m-2}+p^{2m-3}-p^{2m-4}+p^{m-1}-p^{m-2}+2)).$
\item $(i, \lambda)=(p^m-p^{m-1}+p^{\frac{m-1}{2}}, p^{\frac{m-1}{2}}(p^{\frac{m+1}{2}}-p^{\frac{m-1}{2}}+1)(p^m-p^{m-1}+p^{\frac{m-1}{2}}-1)(p^{m+2}-p^{m+1}-p^{m-2}+p^{\frac{m+1}{2}}
    -p^{\frac{m-3}{2}}+p^2)/2(p^2-1)).$
\item $(i, \lambda)=(p^m-p^{m-1}+p^{\frac{m+1}{2}}, p^{\frac{m-3}{2}}(p^{\frac{m-1}{2}}-p^{\frac{m-3}{2}}+1)(p^m-p^{m-1}+p^{\frac{m+1}{2}}-1)(p^{m-1}-1)/2(p^2-1)).$
\item $(i, \lambda)=(p^m-p^{m-1}+p^{\frac{m-1}{2}}(p-1), p^{m-2}(p^{m-1}-1)[p^m-p^{m-1}+p^{\frac{m-1}{2}}(p-1)-1]/2).$
\end{itemize}
For $m=3$, ${\overline{{\mathcal{C}_1}^{\bot}}}^{\bot}$ also holds $2$-$(p^m, i, \lambda)$ designs for the following pairs:
\begin{itemize}
\item $(i, \lambda)=(p^3-2p^2, (p-2)(p-2)(p^3-2p^2-1)/2(p^2-1)).$
\item $(i, \lambda)=(p^3-2p^2+p, p(p^2-1)(p^3-2p^2+p-1)/2).$
\item $(i, \lambda)=(p^3-2p^2-p, p(p^2-p-1)(p^3-p^2-p-1)(p^5-p^4-p+1)/2(p^2-1)).$
\end{itemize}
\end{theorem}

\begin{proof}
By Theorem \ref{design parameter}, one can prove that the number of supports of all codewords with weight $i\neq 0$ in the code ${\overline{{\mathcal{C}_1}^{\bot}}}^{\bot}$ is equal to $A_i/(p-1)$ for each $i,$ where $A_i$ is given in Table \ref{3}. Then the desired conclusions follow from Theorem \ref{$2-$design-1} and  Eq.(\ref{condition}). The proof is then completed.
\end{proof}

\begin{example}\label{example-9}
If $(p,m)=(3,3)$, then the code ${\overline{{\mathcal{C}_1}^{\bot}}}^{\bot}$ has parameters $[27, 10, 9]$ and the weight distribution is given in Example \ref{example1}. It holds $2$-$(27, i, \lambda)$ designs with the following pairs $(i, \lambda):$
$$(9, 4), (12, 132), (15, 2100),$$
which confirms  the results in Theorem \ref{parameter-1}.
\end{example}

\begin{theorem}\label{parameter-2}
Let $m$ be an even integer and $\mathcal{B}$ be the set of the supports of the codewords of ${\overline{{\mathcal{C}_1}^{\bot}}}^{\bot}$ with weight $i,$ where $A_i\neq 0.$ Then for $m\geq 6$, ${\overline{{\mathcal{C}_1}^{\bot}}}^{\bot}$ holds $2$-$(p^m, i, \lambda)$ designs for the following pairs:
\begin{itemize}
\item $(i, \lambda)=(p^m-p^{m-1}, (p^m-p^{m-1}-1)(p^{2m-1}-p^{2m-2}+2p^{2m-3}-p^{m-2}+1)).$
\item $(i, \lambda)=(p^m-p^{m-1}+p^{\frac{m}{2}}-p^{\frac{m-2}{2}}, p^{\frac{m+2}{2}}[p^m-p^{m-1}+p^{\frac{m-2}{2}}(p-1)-1](p^{\frac{m}{2}}+1)(p^m-p^{m-1}-p^{m-2}-p^{\frac{m}{2}}+p^{\frac{m-2}{2}}+1)/2(p^2-1)).$
\item $(i, \lambda)=(p^m-p^{m-1}-p^{\frac{m}{2}}+p^{\frac{m-2}{2}}, p^{\frac{m+2}{2}}[p^m-p^{m-1}-p^{\frac{m-2}{2}}(p-1)-1](p^{\frac{m}{2}}-1)(p^m-p^{m-1}-p^{m-2}+p^{\frac{m}{2}}-p^{\frac{m-2}{2}}+1)/2(p^2-1)).$
\item $(i, \lambda)=(p^m-p^{m-1}-p^{\frac{m+2}{2}}+p^{\frac{m}{2}}, p^{\frac{m}{2}-2}(p^{m-2}-1)(p^{\frac{m}{2}}-1)(p^m-p^{m-1}-p^{\frac{m+2}{2}}+p^{\frac{m}{2}}-1)/2(p^2-1)).$
\item $(i, \lambda)=(p^m-p^{m-1}+p^{\frac{m}{2}-1}, p^{\frac{m+2}{2}}(p^m-p^{m-1}+p^{\frac{m-2}{2}}-1)(p^{\frac{m}{2}+1}-p^{\frac{m}{2}}+1)(p^m-p^{m-1}-p^{m-2}+p^{\frac{m}{2}}-p^{\frac{m-2}{2}}+1)
    /2(p^2-1)).$
\item $(i, \lambda)=(p^m-p^{m-1}+p^{\frac{m}{2}}, p^{\frac{m}{2}-2}(p^m-p^{m-1}+p^{\frac{m}{2}}-1)(p^{\frac{m}{2}}-p^{\frac{m}{2}-1}+1)(p^{m+2}-p^m+p^{m-1}+p^{\frac{m}{2}}-p^{\frac{m}{2}-1}-1)
    /2(p^2-1)).$
\item $(i, \lambda)=(p^m-p^{m-1}-p^{\frac{m}{2}-1}, p^{\frac{m+2}{2}}(p^m-p^{m-1}-p^{\frac{m-2}{2}}-1)(p^{\frac{m}{2}+1}-p^{\frac{m}{2}}-1)(p^m-p^{m-1}-p^{m-2}-p^{\frac{m}{2}}+p^{\frac{m-2}{2}}+1)
    /2(p^2-1)).$
\item $(i, \lambda)=(p^m-p^{m-1}-p^{\frac{m}{2}}, p^{\frac{m}{2}-2}(p^m-p^{m-1}-p^{\frac{m}{2}}-1)(p^{\frac{m}{2}}-p^{\frac{m}{2}-1}-1)(p^{m+2}-p^m+p^{m-1}-p^{\frac{m}{2}}+p^{\frac{m}{2}-1}-1)
    /2(p^2-1)).$
\item $(i, \lambda)=(p^m-p^{m-1}+p^{\frac{m}{2}}(p-1), p^{\frac{m}{2}-2}(p^{m-2}-1)(p^{\frac{m}{2}}+1)[p^m-p^{m-1}+p^{\frac{m}{2}}(p-1)-1]/2(p^2-1)).$
\end{itemize}
Moreover, for $m=4$, ${\overline{{\mathcal{C}_1}^{\bot}}}^{\bot}$ also holds $2$-$(p^4, i, \lambda)$ designs for the first eight pairs as above except for the last one. 
\end{theorem}
%

\begin{example}\label{example-10}
If $(p,m)=(3,4)$, then the code ${\overline{{\mathcal{C}_1}^{\bot}}}^{\bot}$ has parameters $[81, 13, 36]$ and the weight distribution is given in Example \ref{example3}. It gives  $2$-$(81, i, \lambda)$ designs with the following pairs $(i, \lambda):$
$$(60, 39825), (54, 102608), (48, 32994), (36, 140),$$ $$(57, 93366), (63, 18445), (51, 57375),(45, 9185),$$
which confirms  the results in Theorem \ref{parameter-2}.
\end{example}

\begin{theorem}\label{parameter-3}
Let $m\geq 3$ be an integer and $\mathcal{B}$ be the set of the supports of the codewords of ${\overline{{\mathcal{C}_2}^{\bot}}}^{\bot}$ with weight $i,$ where $A_i\neq 0.$ Then for $m\geq 3$ odd, ${\overline{{\mathcal{C}_2}^{\bot}}}^{\bot}$ gives  $2$-$(p^m, i, \lambda)$ designs for the following pairs:
\begin{itemize}
\item $(i, \lambda)=(p^m-p^{m-1}, (p^m-p^{m-1}-1)(p^{m-1}+1)).$
\item $(i, \lambda)=(p^m-p^{m-1}+p^{\frac{m-1}{2}}, \frac{1}{2}p^{\frac{m-1}{2}}(p^{\frac{m+1}{2}}-p^{\frac{m-1}{2}}+1)(p^m-p^{m-1}+p^{\frac{m-1}{2}}-1)).$
\item $(i, \lambda)=(p^m-p^{m-1}-p^{\frac{m-1}{2}}, \frac{1}{2}p^{\frac{m-1}{2}}(p^{\frac{m+1}{2}}-p^{\frac{m-1}{2}}-1)(p^m-p^{m-1}-p^{\frac{m-1}{2}}-1)).$
\end{itemize}
For $m\geq 4$ even, it holds $2$-$(p^m, i, \lambda)$ designs for the following pairs:
\begin{itemize}
\item $(i, \lambda)=(p^m-p^{m-1}, (p^m-p^{m-1}-1)(p^{m-1}-p^{m-2}+1)).$
\item $(i, \lambda)=(p^{\frac{m}{2}-1}(p-1)(p^{\frac{m}{2}}-(-1)^{\frac{m}{2}}), p^{\frac{m}{2}}(p^{\frac{m}{2}}-(-1)^{\frac{m}{2}})[p^{\frac{m}{2}-1}(p-1)(p^{\frac{m}{2}}-(-1)^{\frac{m}{2}})-1]/(p+1)).$
\item $(i, \lambda)=(p^{\frac{m}{2}-1}[p^{\frac{m}{2}}(p-1)+(-1)^{\frac{m}{2}}], p^{\frac{m}{2}}[p^{\frac{m}{2}}(p-1)+(-1)^{\frac{m}{2}}][p^{m-1}(p-1)+(-1)^{\frac{m}{2}}p^{\frac{m}{2}-1}-1]/(p+1)).$
\item $(i, \lambda)=(p^{\frac{m}{2}}(p-1)[p^{\frac{m}{2}-1}+(-1)^{\frac{m}{2}}], p^{\frac{m}{2}-2}(p^{\frac{m}{2}-1}+(-1)^{\frac{m}{2}}][p^{\frac{m}{2}}(p-1)(p^{\frac{m}{2}-1}+(-1)^{\frac{m}{2}})-1]/(p+1)).$
\item $(i, \lambda)=(p^{\frac{m}{2}}[p^{\frac{m}{2}-1}(p-1)-(-1)^{\frac{m}{2}}], p^{\frac{m}{2}-2}[p^{\frac{m}{2}-1}(p-1)-(-1)^{\frac{m}{2}}][p^{m-1}(p-1)-(-1)^{\frac{m}{2}}p^{\frac{m}{2}}-1]/(p+1)).$
\end{itemize}
\end{theorem}


\begin{example}\label{example-11}
If $(p,m)=(3,3)$, then the code ${\overline{{\mathcal{C}_2}^{\bot}}}^{\bot}$ has parameters $[27, 7, 15]$. It gives $2$-$(27, i, \lambda)$ designs with the following pairs $(i, \lambda):$
$$(15, 105), (18, 170), (21, 210),$$
which confirms  the results in Theorem \ref{parameter-3}.
\end{example}

\begin{example}\label{example-12}
If $(p,m)=(3,4)$, then the code ${\overline{{\mathcal{C}_2}^{\bot}}}^{\bot}$ has parameters $[81, 9, 45]$ and the weight distribution is given in Example \ref{example7}. It gives  $2$-$(81, i, \lambda)$ designs with the following pairs $(i, \lambda):$
$$(45, 55), (54, 1007), (48, 846), (57, 2394), (72, 71),$$
which confirms  the results in Theorem \ref{parameter-3}.
\end{example}

\section{Concluding remarks}\label{section-5}
In this paper, we first determined the weight distributions of two classes of linear codes derived from the duals of extended cyclic codes. Using the properties of affine-invariant codes, we then found that both ${\overline{{\mathcal{C}_1}^{\bot}}}^{\bot}$ and ${\overline{{\mathcal{C}_2}^{\bot}}}^{\bot}$  hold $2$-designs and explicitly determined their parameters. However, for ${\overline{{\mathcal{C}_1}^{\bot}}}^{\bot}$, it seems hard to determine the parameters of the  $2$-designs  derived from the supports of all codewords with weight $i=p^2(p-1), p(p^2-p-1), p^3, p(p^2-1)$ for $m=3,$  and weight $i=p^2(p^2-1)$ for $m=4$, respectively.
This may constitute a challenge for future work.

\section*{Appendix I}
\begin{table}
\begin{center}
\caption{The value distribution of $S(a,b,c)$ when $m$ is odd}\label{1}
\begin{tabular}{ll}
\hline\noalign{\smallskip}
Value  &  Multiplicity   \\
\noalign{\smallskip}
\hline\noalign{\smallskip}
$\sqrt{p^*}p^{\frac{m-1}{2}}, -\sqrt{p^*}p^{\frac{m-1}{2}}$  &  $ \frac{1}{2}p^{m+1}(p^m-p^{m-1}-p^{m-2}+1)(p^m-1)/(p^2-1)$ \\
$\zeta^j_p\sqrt{p^*}p^{\frac{m-1}{2}}$, for $ 1\leq j \leq p-1$  &  $ \frac{1}{2}p^{\frac{m+3}{2}}(p^{\frac{m-1}{2}}+(\frac{-j}{p}))(p^m-p^{m-1}-p^{m-2}+1)\frac{p^m-1}{p^2-1} $    \\
$ -\zeta^j_p\sqrt{p^*}p^{\frac{m-1}{2}}$, for $ 1\leq j \leq p-1 $  &  $ \frac{1}{2}p^{\frac{m+3}{2}}(p^{\frac{m-1}{2}}-(\frac{-j}{p}))(p^m-p^{m-1}-p^{m-2}+1)\frac{p^m-1}{p^2-1} $     \\
$ p^{\frac{m+1}{2}} $  &  $n_{1,1,0}=\frac{1}{2}p^{m-2}(p^{\frac{m-1}{2}}+1)(p^{\frac{m-1}{2}}+p-1)(p^m-1)$     \\
$  -p^{\frac{m+1}{2}}  $  &  $n_{-1,1,0}=\frac{1}{2}p^{m-2}(p^{\frac{m-1}{2}}-1)(p^{\frac{m-1}{2}}-p+1)(p^m-1) $     \\
$ \zeta^j_pp^{\frac{m+1}{2}}, -\zeta^j_pp^{\frac{m+1}{2}}$, for $ 1\leq j \leq p-1$  &  $ n_{\pm1,1,1}=\frac{1}{2}p^{m-2}(p^{m-1}-1)(p^m-1) $  \\
$\sqrt{p^*}p^{\frac{m+1}{2}}, -\sqrt{p^*}p^{\frac{m+1}{2}}$ &  $ n_{\pm1,2,0}=\frac{1}{2}p^{m-3}(p^{m-1}-1)(p^m-1)/(p^2-1)$  \\
$\zeta^j_p\sqrt{p^*}p^{\frac{m+1}{2}}$, for $ 1\leq j \leq p-1$ &  $ n_{1,2,1}=\frac{1}{2}p^{\frac{m-3}{2}}(p^{\frac{m-3}{2}}+(\frac{-j}{p}))(p^{m-1}-1)\frac{p^m-1}{p^2-1}$  \\
$-\zeta^j_p\sqrt{p^*}p^{\frac{m+1}{2}}$, for $ 1\leq j \leq p-1$ &  $ n_{-1,2,1}=\frac{1}{2}p^{\frac{m-3}{2}}(p^{\frac{m-3}{2}}-(\frac{-j}{p}))(p^{m-1}-1)\frac{p^m-1}{p^2-1} $  \\
$0$ &  $ w=(p^m-1)(p^{2m-1}-p^{2m-2}+p^{2m-3}-p^{m-2}+1) $  \\
$p^m$ &  $ n_{p^m}=1 $  \\
\noalign{\smallskip}
\hline
\end{tabular}
\end{center}
\end{table}

\begin{table}
\begin{center}
\caption{The value distribution of $S(a,b,c)$ when $m$ is even}\label{2}
\begin{tabular}{ll}
\hline\noalign{\smallskip}
Value  &  Multiplicity   \\
\noalign{\smallskip}
\hline\noalign{\smallskip}
$p^{\frac{m}{2}}$  &  $\frac{(p^{\frac{m}{2}}+p-1)(p^m-p^{m-1}-p^{m-2}+p^{\frac{m}{2}}-p^{\frac{m}{2}-1}+1)p^{\frac{m}{2}+1}(p^m-1)}{2(p^2-1)}$ \\
$-p^{\frac{m}{2}}$  &  $\frac{ p^{\frac{m}{2}+1}(p^{\frac{m}{2}}-p+1)(p^m-p^{m-1}-p^{m-2}-p^{\frac{m}{2}}+p^{\frac{m}{2}-1}+1)(p^m-1)}{2(p^2-1)}$\\
$ \zeta^j_p p^{\frac{m}{2}}$, for $ 1\leq j \leq p-1 $  &  $ \frac{p^{\frac{m}{2}+1}(p^{\frac{m}{2}}-1)(p^m-p^{m-1}-p^{m-2}+p^{\frac{m}{2}}-p^{\frac{m}{2}-1}+1)(p^m-1)}{2(p^2-1)} $     \\
$ -\zeta^j_p p^{\frac{m}{2}}$, for $ 1\leq j \leq p-1 $  &  $\frac{p^{\frac{m}{2}+1}(p^{\frac{m}{2}}+1)(p^m-p^{m-1}-p^{m-2}-p^{\frac{m}{2}}+p^{\frac{m}{2}-1}+1)(p^m-1)}{2(p^2-1)}$     \\
$  \sqrt{p^*}p^{\frac{m}{2}}, -\sqrt{p^*}p^{\frac{m}{2}} $  &  $n_{\pm1,1,0}=\frac{1}{2}p^{2m-3}(p^m-1) $     \\
$ \zeta^j_p\sqrt{p^*}p^{\frac{m}{2}}$, for $ 1\leq j \leq p-1$  &  $ n_{1,1,1}=\frac{1}{2}p^{\frac{3}{2}m-2}(p^{\frac{m}{2}-1}+(\frac{-j}{p}))(p^m-1) $     \\
$-\zeta^j_p\sqrt{p^*}p^{\frac{m}{2}}$, for $ 1\leq j \leq p-1$ &  $n_{-1,1,1}=\frac{1}{2}p^{\frac{3}{2}m-2}(p^{\frac{m}{2}-1}-(\frac{-j}{p}))(p^m-1)) $  \\
$p^{\frac{m}{2}+1}$ &  $\frac{1}{2}p^{\frac{m}{2}-2}(p^{\frac{m}{2}-1}+1)(p^{\frac{m}{2}}-1)(p^{\frac{m}{2}-1}+p-1)\frac{p^m-1}{p^2-1}$\\
$-p^{\frac{m}{2}+1}$ &  $ \frac{1}{2}p^{\frac{m}{2}-2}(p^{\frac{m}{2}-1}-1)(p^{\frac{m}{2}}+1)(p^{\frac{m}{2}-1}-p+1)\frac{p^m-1}{p^2-1}$  \\
$\zeta^j_pp^{\frac{m}{2}+1}$, for $ 1\leq j \leq p-1$ &  $\frac{1}{2}p^{\frac{m}{2}-2}(p^{\frac{m}{2}}-1)(p^{m-2}-1)(p^m-1)/(p^2-1) $ \\
$-\zeta^j_pp^{\frac{m}{2}+1}$, for $ 1\leq j \leq p-1$ &  $ \frac{1}{2}p^{\frac{m}{2}-2}(p^{\frac{m}{2}}+1)(p^{m-2}-1)(p^m-1)/(p^2-1)$\\
$0$ &  $ w=(p^m-1)(p^{2m-1}-p^{2m-2}+p^{2m-3}-p^{m-2}+1) $  \\
$p^m$ &  $n_{p^m}=1 $  \\
\noalign{\smallskip}
\hline
\end{tabular}
\end{center}
\end{table}

\section*{Appendix II}


\begin{proof}[Proof of Theorem \ref{weight1}]
For each nonzero codeword $\mathbf{c}(a,b,c,h)=(c_0, c_1,\ldots, c_n)$ in ${\overline{{\mathcal{C}_1}^{\bot}}}^{\bot},$ the Hamming weight of $\mathbf{c}(a,b,c,h)$ is
\begin{equation}\label{weight formula-1}
w_H(\mathbf{c}(a,b,c,h))=n+1-T(a,b,c,h)=p^m-T(a,b,c,h),
\end{equation}
where $$T(a,b,c,h)=|\{x: Tr(ax^{p^l+1}+bx^2+cx)+h=0,~ x,a,b,c \in \mathbb{F}_q, h\in \mathbb{F}_p\}|.$$
Then
\begin{eqnarray*}
T(a,b,c,h)&=&\frac{1}{p}\sum\limits_{y \in \mathbb{F}_p}\sum\limits_{x \in \mathbb{F}_q}\zeta_p^{y(Tr(ax^{p^l+1}+bx^2+cx)+h)}\\
&=&\frac{1}{p}\sum\limits_{y \in \mathbb{F}_p}\zeta_p^{yh}\sum\limits_{x \in \mathbb{F}_q}\zeta_p^{yTr(ax^{p^l+1}+bx^2+cx)}\\
&=&p^{m-1}+\frac{1}{p}\sum\limits_{y \in \mathbb{F}_p^*}\zeta_p^{yh}\sigma_y(S(a,b,c)).
\end{eqnarray*}
By Lemma \ref{value distribution}, for $\varepsilon=\pm 1, j \in \mathbb{F}_p^*$ and $0\leq i\leq 2,$ we have $\ell=\frac{m+i}{2}$ if $m-i$ is even, and $ \ell=\frac{m+i-1}{2}$ if $m-i$ is odd, and then
$$S(a,b,c)=\{\varepsilon p^\ell, \varepsilon \sqrt{p^*}p^\ell, 0 , \varepsilon p^\ell\zeta_p^j, \varepsilon \sqrt{p^*}\zeta_p^jp^\ell, p^m\}.$$
Hence from Lemma \ref{Cyclotomic fields}, we get
\begin{eqnarray*}
\sigma_y(S(a,b,c))=\left\{
\begin{array}{ll}
0 & \mathrm{if}\,\ S(a,b,c)=0 ,\\
\varepsilon p^\ell & \mathrm{if}\,\ S(a,b,c)=\varepsilon p^\ell ,\\
\varepsilon p^\ell \sqrt{p^*}\eta'(y) & \mathrm{if}\,\ S(a,b,c)=\varepsilon p^\ell \sqrt{p^*} ,\\
\varepsilon p^\ell \zeta_p^{yj} & \mathrm{if}\,\ S(a,b,c)=\varepsilon p^\ell\zeta_p^j ,\\
\varepsilon p^\ell \sqrt{p^*}\eta'(y) \zeta_p^{yj}& \mathrm{if}\,\ S(a,b,c)=\varepsilon\sqrt{p^*}p^\ell\zeta_p^j ,\\
p^m & \mathrm{if}\,\ S(a,b,c)=p^m.
\end{array}
\right.
\end{eqnarray*}
That is,
\begin{eqnarray*}
T(a,b,c,h)=\left\{
\begin{array}{ll}
p^{m-1} & \mathrm{if}\,\ S(a,b,c)=0 ,\\
p^{m-1}+\varepsilon p^{\ell -1}(p-1) & \mathrm{if}\,\ S(a,b,c)=\varepsilon p^\ell \, \mathrm{and}\,\ h=0,\\
p^{m-1}+\varepsilon p^{\ell -1}(-1) & \mathrm{if}\,\ S(a,b,c) =\varepsilon p^\ell\, \mathrm{and}\,\ h\neq0 ,\\
p^{m-1}+\varepsilon p^{\ell -1}\sqrt{p^*}\eta'(y)G(\eta', \chi'_1) & \mathrm{if}\,\ S(a,b,c)=\varepsilon \sqrt{p^*}p^\ell ,\\
p^{m-1}+\varepsilon p^{\ell -1}(p-1)& \mathrm{if}\,\ S(a,b,c)=\varepsilon p^\ell \zeta_p^j\, \mathrm{and}\,\ h+j=0, \\
p^{m-1}+\varepsilon p^{\ell -1}(-1)& \mathrm{if}\,\ S(a,b,c)=\varepsilon p^\ell \zeta_p^j \, \mathrm{and}\,\ h+j\neq 0, \\
p^{m-1}+\varepsilon p^{\ell -1}\sqrt{p^*}\eta(h+j)G(\eta',\chi'_1) & \mathrm{if}\,\ S(a,b,c)=\varepsilon \sqrt{p^*}p^\ell\zeta_p^j ,\\
p^m& \mathrm{if}\,\ S(a,b,c)=p^m \, \mathrm{and}\,\ h=0 ,\\
0& \mathrm{if}\,\ S(a,b,c)=p^m\, \mathrm{and}\,\ h \neq 0 .
\end{array}
\right.
\end{eqnarray*}
Obviously, when $m$ is odd, by Lemmas \ref{Gauss}, \ref{Qu-Fea}, \ref{value distribution} and Eq.(\ref{weight formula-1}) we have

$w_1=p^m-p^{m-1},$

$A_{w_1}=pw+2n_{1,0,0}+2n_{1,2,0}+(p-1)[(n_{1,0,1}+n_{-1,0,1})+(n_{1,2,1}+n_{-1,2,1})],$

$w_2=p^m-[p^{m-1}+p^{\frac{m-1}{2}}(p-1)],$

$A_{w_2}=n_{1,1,0}+(p-1)n_{1,1,1},$

$w_3=p^m-[p^{m-1}-p^{\frac{m-1}{2}}(p-1)],$

$A_{w_3}=n_{-1,1,0}+(p-1)n_{-1,1,1},$

$w_4=p^m-(p^{m-1}-p^{\frac{m-1}{2}}),$

$A_{w_4}=(p-1)n_{1,1,0}+(p-1)^2n_{1,1,1}+{\frac{p-1}{2}}(n_{1,0,0}+n_{-1,0,0})+{\frac{p-1}{2}}{(p-1)}(n_{1,0,1}+n_{-1,0,1}),$

$w_5=p^m-(p^{m-1}+p^{\frac{m-1}{2}}),$

$A_{w_5}=(p-1)n_{-1,1,0}+(p-1)^2n_{-1,1,1}+2\cdot{\frac{p-1}{2}}n_{\pm1,0,0}+2\cdot{\frac{(p-1)^2}{2}}n_{\pm1,0,1},$

$w_6=p^m-(p^{m-1}-p^{\frac{m+1}{2}}),$

$A_{w_6}=2\cdot{\frac{p-1}{2}}n_{\pm1,2,0}+2\cdot{\frac{(p-1)^2}{2}}n_{\pm1,2,1},$

$w_7=p^m-(p^{m-1}+p^{\frac{m+1}{2}}),$

$A_{w_7}=A_{w_6},$

$A_{p^m}=p-1.$\\
When $m$ is even, by Lemmas \ref{Gauss}, \ref{Qu-Fea}, \ref{value distribution} and Eq.(\ref{weight formula-1}) we get

$w_1=p^m-p^{m-1},$

$A_{w_1}=pw+2n_{\pm1,1,0}+2(p-1)n_{\pm1,1,1},$

$w_2=p^m-[p^{m-1}+p^{\frac{m}{2}-1}(p-1)],$

$A_{w_2}=n_{1,0,0}+(p-1)n_{1,0,1},$

$w_3=p^m-[p^{m-1}+p^{\frac{m}{2}}(p-1)],$

$A_{w_3}=n_{1,2,0}+(p-1)n_{1,2,1},$

$w_4=p^m-(p^{m-1}-p^{\frac{m}{2}-1}(p-1)),$

$A_{w_4}=n_{-1,0,0}+(p-1)n_{-1,0,1},$

$w_5=p^m-[p^{m-1}-p^{\frac{m}{2}}(p-1)],$

$A_{w_5}=n_{-1,2,0}+(p-1)n_{-1,2,1},$

$w_6=p^m-[p^{m-1}-p^{\frac{m}{2}}(p-1)],$

$A_{w_6}=(p-1)n_{1,0,0}+(p-1)^2n_{1,0,1},$

$w_7=p^m-(p^{m-1}-p^{\frac{m}{2}}),$

$A_{w_7}=(p-1)n_{1,2,0}+(p-1)^2n_{1,2,1}+2\cdot{\frac{p-1}{2}}n_{\pm1,1,0}+2\cdot{\frac{(p-1)^2}{2}}n_{\pm1,1,1},$

$w_8=p^m-(p^{m-1}+p^{\frac{m}{2}-1}),$

$A_{w_8}=(p-1)n_{-1,0,0}+(p-1)^2n_{-,0,1},$

$w_9=p^m-(p^{m-1}+p^{\frac{m}{2}}),$

$A_{w_9}=(p-1)n_{-1,2,0}+(p-1)^2n_{-1,2,1}+2\cdot{\frac{p-1}{2}}n_{\pm1,1,0}+2\cdot{\frac{(p-1)^2}{2}}n_{\pm1,1,1},$

$A_{p^m}=p-1.$\\
Thus we complete the proof of Theorem \ref{weight1}.
\end{proof}


\begin{proof}[Proof of Theorem \ref{weight2}]
For each nonzero codeword $\mathbf{c}(a,b,h)=(c_0,\ldots,c_n)$ in ${\overline{{\mathcal{C}_2}^{\bot}}}^{\bot},$ the Hamming weight of $\mathbf{c}(a,b,h)$ is
\begin{equation}\label{weight formula-2}
w_H(\mathbf{c}(a,b,h))=p^m-T(a,b,h),
\end{equation}
where $$T(a,b,h)=|\{x:Tr(ax^{p^l+1}+bx)+h=0,~ x,a,b \in \mathbb{F}_q, h \in \mathbb{F}_p\}|.$$
Then
\begin{eqnarray*}
T(a,b,h)&=&{\frac{1}{p}}\sum\limits_{y \in \mathbb{F}_p}\sum\limits_{x \in \mathbb{F}_q}\zeta_p^{yTr(ax^{p^l+1}+bx)+hy}\\
&=&p^{m-1}+{\frac{1}{p}}\sum\limits_{y\in \mathbb{F}_p^*}\zeta_p^{hy}\sum\limits_{x \in \mathbb{F}_q}\zeta_p^{yTr(ax^{p^l+1}+bx)}.
\end{eqnarray*}
If $a=b=h=0$, then $\mathbf{c}(a,b,h)$ is the zero codeword.\\
If $a=b=0, h\neq 0$, then  $T(a,b,h)=p^{m-1}+p^{m-1}\sum\limits_{y \in \mathbb{F}_p^*}\zeta_p^{hy}=0.$\\
If $a=0, b\neq 0$, then $ T(a,b,h)=p^{m-1}+{\frac{1}{p}}\sum\limits_{y \in \mathbb{F}_p^*}\zeta_p^{hy}\sum\limits_{x\in \mathbb{F}_q}\zeta_p^{yTr(bx)}=p^{m-1}.$\\
If $a\neq 0$, then
\begin{eqnarray*}
T(a,b,h)&=&p^{m-1}+{\frac{1}{p}}\sum\limits_{y \in \mathbb{F}_p^*}\zeta_p^{hy}\sigma_y(\sum\limits_{x\in \mathbb{F}_q}\zeta_p^{Tr(ax^{p^l+1}+bx)})\\
&=&p^{m-1}+{\frac{1}{p}}\sum\limits_{y \in \mathbb{F}_p^*}\zeta_p^{hy}\sigma_y(S(a,b))\\
&=&p^{m-1}+{\frac{1}{p}}\sum\limits_{y \in \mathbb{F}_p^*}\zeta_p^{hy}S(ay,by).
\end{eqnarray*}

When $m$ is odd, by Lemmas \ref{Gauss}-\ref{polynomial}, we have
\begin{eqnarray*}
T(a,b,h)=\left\{
\begin{array}{ll}
p^{m-1} & \mathrm{if}\,\ h=Tr(ax_{a,b}^{p^l+1}),\\
p^{m-1}+p^{\frac{m-1}{2}}(-1)^\frac{(p-1)(m+1)}{4}\eta(a)\eta'(h-Tr(ax_{a,b}^{p^l+1})) & \mathrm{if}\,\ h\neq Tr(ax_{a,b}^{p^l+1}).\\
\end{array}
\right.
\end{eqnarray*}
Obviously, for $ a\in \mathbb{F}_q^*,$ we get that  $ T(a,b,h)=p^{m-1}$ appears $p^m(p^m-1)$ times, both $T(a,b,h)=p^{m-1}+p^{\frac{m-1}{2}}$ and $ T(a,b,h)=p^{m-1}-p^{\frac{m-1}{2}}$  appear $\frac{p^m(p-1)(p^m-1)}{2}$ times, respectively.

When $m$ is even and $a^{\frac{q-1}{p+1}}\neq (-1)^{\frac{m}{2}},$ by Lemma \ref{polynomial}, we have
\begin{eqnarray*}
T(a,b,h)=\left\{
\begin{array}{ll}
p^{m-1}+(-1)^{\frac{m}{2}}p^{\frac{m}{2}-1}(p-1) & \mathrm{if}\,\ h=Tr(ax_{a,b}^{p^l+1}),\\
p^{m-1}-(-1)^{\frac{m}{2}}p^{\frac{m}{2}-1} & \mathrm{if}\,\  h\neq Tr(ax_{a,b}^{p^l+1})\,.
\end{array}
\right.
\end{eqnarray*}
Clearly, there exist $p^m-1-{\frac{p^m-1}{p+1}}={\frac{p(p^m-1)}{p+1}}$ elements $a\in \mathbb{F}_q^*$ such that
$a^{\frac{q-1}{p+1}}\neq (-1)^{\frac{m}{2}}.$ Then
$$T(a,b,h)=p^{m-1}+(-1)^{\frac{m}{2}}p^{\frac{m}{2}-1}(p-1)$$
appears ${\frac{p^{m+1}(p^m-1)}{p+1}}$ times, and
$$T(a,b,h)=p^{m-1}-(-1)^{\frac{m}{2}}p^{\frac{m}{2}-1}$$
appears ${\frac{p^{m+1}(p-1)(p^m-1)}{p+1}}$ times.

When $m$ is even and $a^{\frac{q-1}{p+1}}=(-1)^{\frac{m}{2}},$ from Lemma \ref{polynomial}, we get
\begin{eqnarray*}
T(a,b,h)=\left\{
\begin{array}{ll}
p^{m-1}-(-1)^{\frac{m}{2}}p^{\frac{m}{2}}(p-1) & \mathrm{if}\,\ f(x)=-b^{p^l}\, \ \mathrm{is~ solvable}\, \mathrm{and}\,\ h=Tr(ax_{a,b}^{p^l+1}),\, \\
p^{m-1}+(-1)^{\frac{m}{2}}p^{\frac{m}{2}} & \mathrm{if}\,\ f(x)=-b^{p^l}\, \ \mathrm{is ~solvable},\,\ b\neq 0\,\mathrm{and}\\&\ h\neq Tr(ax_{a,b}^{p^l+1})\ ,\ \mathrm{or}\,\ b=0\,\mathrm{and}\,\ h\neq 0,\\
p^{m-1}  & \mathrm{if}\,\ f(x)=-b^{p^l}\, \ \mathrm{is ~ no~ solvable}.
\end{array}
\right.
\end{eqnarray*}
By Lemma \ref{solution}, there are $\frac{q-1}{p+1}$ elements $a\in \mathbb{F}_q^*$ such that $a^{\frac{q-1}{p+1}}=(-1)^{\frac{m}{2}},$ and $p^{m-2}$ elements $b\in \mathbb{F}_q$ such that $f(x)=-b^{p^l}$ is solvable. Therefore,
$$T(a,b,h)=p^{m-1}+(-1)^{\frac{m}{2}+1}p^{\frac{m}{2}}(p-1)$$
 appears ${\frac{p^{m-2}(p^m-1)}{p+1}}$ times,
 $$T(a,b,h)=p^{m-1}+(-1)^{\frac{m}{2}}p^{\frac{m}{2}}$$
 appears ${\frac{p^{m-2}(p^m-1)(p-1)}{p+1}}$ times, and $T(a,b,h)=p^{m-1}$ appears $(p^m-1)p^{m-1}(p-1)$ times.

By all the discussions above, the proof is completed.
\end{proof}



\end{document}